\documentclass[10pt]{article}

\usepackage{amsfonts,a4wide,amsmath,amssymb,amsthm}
\usepackage[utf8]{inputenc}  
\usepackage[T1]{fontenc}
\usepackage[all]{xy}    
\usepackage{mathtools}  
\usepackage{multirow}    
\usepackage{csquotes}
\usepackage{hyperref}      
\hypersetup{
     colorlinks=true, 
     breaklinks=true, 
     urlcolor= blue,  
     linkcolor= blue, 
     linktoc	=all,	
     citecolor=red,	
     filecolor=black,	
     bookmarksopen=true,            
     pdftitle={Nonvanishing of critical values of L-functions}, 
     pdfauthor={Samuel Le Fourn},     
}
\usepackage{tikz}   

\theoremstyle{plain}
\newtheorem{thm}{Theorem}
\newtheorem{prop}[thm]{Proposition}
\newtheorem{lem}[thm]{Lemma}
\newtheorem{cor}[thm]{Corollary}

\theoremstyle{definition}
\newtheorem{defi}[thm]{Definition}

\theoremstyle{remark}
 \newtheorem{rem}[thm]{Remark}

\newcommand{\Bcal}{{\mathcal B}}
\newcommand{\Ccal}{{\mathcal C}}
\newcommand{\Dcal}{{\mathcal D}}

\newcommand{\Fcal}{{\mathcal F}}

\newcommand{\Hcal}{{\mathcal H}}

\newcommand{\Rcal}{{\mathcal R}}

\newcommand{\N}{{\mathbb{N}}}
\newcommand{\Z}{{\mathbb{Z}}}
\newcommand{\Q}{{\mathbb{Q}}}
\newcommand{\R}{{\mathbb{R}}}
\newcommand{\C}{{\mathbb{C}}}

\newcommand{\GL}{\operatorname{GL}}

\renewcommand{\Im}{\operatorname{Im}}

\renewcommand{\Re}{\operatorname{Re}}

\newcommand{\Vect}{\operatorname{Vect}}

\newcommand{\mt}{\mapsto}	

\newcommand{\ra}{\rightarrow}

\newcommand{\quot}[2]                                    
{\raisebox{.6ex}{\newline$#1$}\!/\!\raisebox{-.6ex}{$#2$}}

\newcommand{\xonsdp}{X_0 ^{{\textrm{ns}}} (d;p)}

\newcommand{\matabcd}{\begin{pmatrix} a & b \\ c & d \end{pmatrix}}

\title{\Huge Nonvanishing of central values of $L$-functions of newforms in $S_2 (\Gamma_0 (dp^2))$ twisted by quadratic characters }
\author{Samuel Le Fourn\footnote{Email : \url {samuel.le_fourn@ens-lyon.fr}  }\\
ENS de Lyon}
\date\today

\begin{document}
\maketitle

\begin{abstract}
We prove that for $d \in \{ 2,3,5,7,13 \}$ and $K$ a quadratic (or rational) field of discriminant $D$ and Dirichlet character $\chi$, if a prime $p$ is large enough compared to $D$, there is a newform $f \in S_2(\Gamma_0(dp^2))$ with sign $(+1)$ with respect to the Atkin-Lehner involution $w_{p^2}$ such that $L(f \otimes \chi,1) \neq 0$. This result is obtained through an estimate of a weighted sum of twists of $L$-functions which generalises a result of Ellenberg. It relies on the approximate functional equation for the $L$-functions $L(f \otimes \chi, \cdot)$ and a Petersson trace formula restricted to Atkin-Lehner eigenspaces. An application of this nonvanishing theorem will be given in terms of existence of rank zero quotients of some twisted jacobians, which generalises a result of Darmon and Merel.
\smallskip

\smallskip

\textbf{MSC201} \, \, 11F67 and 14J15.
\end{abstract}

\section{Introduction}

Let $d \in \{2,3,5,7,13\}$ (that is, a prime number such that the genus of $X_0(d)$ is zero) and $K/\Q$ be a real extension of degree at most two, with discriminant $D$ assumed prime to $d$ and associated Dirichlet character $\chi$.

The main result of this paper is the following (with the usual notations, recalled in section \ref{sectionnotations}).
 \begin{thm}
 \label{thmmain}
For every prime number $p$ not dividing $dD$,
 	\[
 		\sum_{f \in \Bcal_p} \overline{a_1 (f)} L(f \otimes \chi,1) = 2 \pi + \mathcal{O} \left( \frac{\sqrt{D}(\log(D)+1)^3 \log(p)^2}{p^2} \right),
 	\]
where $\Bcal_p$ is a Petersson-orthonormal basis of the subspace of $S_2(\Gamma_0(dp^2))^{\textrm{new}}$ spanned by the modular forms fixed by the Atkin-Lehner involution $w_{p^2}$, and the implied constant is absolute.
 \end{thm}
 
 \begin{rem}
 \label{remtheorem}
 \hspace*{\fill}

 $(a)$ For $d=1$ and $K$ imaginary, an analogous result (with sign $(-1)$ for $w_{p^2}$ and $4 \pi$ instead of $2 \pi$) is given in \cite{Ellenberg04} (Lemma 3.10 and the associated discussion therein), itself based on the estimates given in (\cite{Ellenberg05}, Theorem 1 and Corollary 2). We gave by similar methods a completely explicit bound for that case in \cite{LeFourn1} (Proposition 5.13 of the Appendix and its proof), but for simplicity we focus here on giving good exponents in $D$ and $p$, without explicit constants.

$(b)$ The method here is the same in principle as in \cite{Ellenberg04}, but the level $dp^2$ instead of $p^2$ involves more preliminary computations to deal with the contribution of the oldforms, which will be done in Lemma \ref{a1Lchinewdp2}. This is also why $d$ is assumed in $\{2,3,5,7,13\}$ : for now, it is the only situation where we can take out the $p^2$-old part, because $S_2(\Gamma_0(d))=0$ (see Remark \ref{remmarchepasautrescas} for further details).

$(c)$ We assume that $D$ and $dp^2$ are coprime to have a simple expression for $(f \otimes \chi)_{|w_{dp^2D^2}}$ in terms of $f_{|w_{dp^2}}$ (Lemma \ref{lemsumLfchi} $(a)$). If it is not the case, $f \otimes \chi$ is a modular form of smaller level (\cite{IwaniecKowalski}, Proposition 14.19), but the action of the (smaller level) Atkin-Lehner involution seems less natural to describe.
 \end{rem}

 As an essential tool for the proof of Theorem \ref{thmmain}, we devised a new Petersson trace formula which essentially gives a closed expression for the same weighted sum as classical Petersson trace formula on $S_2(\Gamma_0(N))$, but for the sum only on the eigenforms having prescribed eigenvalues (1 or $-1$) for the possible Atkin-Lehner involutions on $\Gamma_0(N)$. The general version is given in Proposition \ref{propformulesdestraceslaplusgenerale}, but for now, we only give the case of one prescribed eigenvalue in Proposition \ref{propformulesdestraces} below.
 

\begin{prop}[Restricted Petersson trace formula] 
\label{propformulesdestraces}
\hspace*{\fill}

Let $m,n,N$ be three positive integers, an integer $Q|N$ such that $Q>1$ with $(Q,N/Q)=1$, and $\varepsilon = \pm 1$. Let $\Bcal$ be an eigenbasis of $S_2(\Gamma_0(N))$. Then, we have 
 \begin{eqnarray}
\label{formuledestracesgeneralisee}
  \frac{1}{2 \pi \sqrt{mn}} \sum_{\substack{f \in \Bcal \\ f_{|w_Q} = \varepsilon \cdot f}} \frac{\overline{a_m(f)} a_n(f)}{\|f \|^2}   =  \delta_{mn} &  - &  2 \pi \sum_{\substack{c >0 \\ N |c} } \frac{S(m,n ;c)}{c} J_1 \left( \frac{4 \pi \sqrt{mn}}{c} \right)
 \\ & -  & 2 \pi \varepsilon \sum_{\substack{c > 0 \\ (N/Q) |c \\ (c,Q)= 1}} \frac{ S(m,nQ^{-1};c)}{c \sqrt{Q}} J_1 \left( \frac{4 \pi \sqrt{mn}}{c \sqrt{Q}} \right) \nonumber,
 \end{eqnarray}
 where $nQ^{-1}$ in the Kloosterman sum means  $n$ times the inverse of $Q$ modulo $c$ (see section \ref{sectionPeterssonrestreint} for the definitions of Kloosterman sums and the Bessel function $J_1$).
\end{prop}

Notice there are analogues of Proposition \ref{propformulesdestraces} already given in the litterature, but this version seems to be the most general and the easiest to use (the comparison is made in Remark \ref{remformuletraces}).
 
 The arithmetic motivation of Theorem \ref{thmmain} is the following.
 
 \begin{cor}
 \label{cormain}
	For $d \in \{2,3,5,7,13\}$ and $p \neq d$ a prime number, let 
\[
	J(d,p):=J_0(dp^2)^{p-\textrm{new}} / (w_{p^2} - 1) J_0 (dp^2)^{p-\textrm{new}}.	
\]	

	$(a)$ For $p \gg 1$, there exists a nonzero quotient abelian variety of $J(d,p)$ with only finitely many $\Q$-rational points (such a  quotient is called a `` rank zero quotient '').
	
	$(b)$ Let $K$ be a real quadratic field with discriminant $D$ prime to $d$ and Dirichlet character $\chi$. Let $J(d,p,\chi)$ be the twist of $J(d,p)$ by the extension $K/\Q$ relatively to the automorphism $[\chi(d)] w_d$, so that the points of $J(d,p,\chi)(\Q)$ correspond to the points $P$ of $J(d,p)(K)$ such that
	\[
		P^{\sigma} = \chi(d) w_d \cdot P,
	\]
where $\sigma$ is the nontrivial automorphism of $K$. 
Then, for $p \gg D^{1/4} \log(D)^3$, there is a rank zero quotient of $J(d,p,\chi)$.

$(c)$ The same results hold when replacing $J(d,p)$ by the jacobian $J'(d,p)$ of the modular curve $\xonsdp$ parametrising the triples $(E,C_d,\alpha_p)$, where $C_d$ is a cyclic subgroup of $E$ of order $d$, and $\alpha_p$ a \enquote{normaliser of nonsplit Cartan structure} on $E$. of the modular curve. To be precise, the twist $J'(d,p,\chi)$ of $J'(d,p)$ by the extension $K/\Q$ relatively to the automorphism $[\chi(d)] w_d$ also has a rank zero quotient for $p \gg D^{1/4} \log(D)^3$.

\end{cor}

The proof of Corollary \ref{cormain} (which is quite straightforward for the reader familiar with the techniques), has been displaced to the end of this paper to keep the focus on the analytic side for now. The following remark compares it with others results in the litterature.

\begin{rem}
\hspace*{\fill}

$(a)$ The part $(a)$ of Corollary \ref{cormain} has been already proven in \cite{DarmonMerel} through modular symbols for $p \geq 7$. Their result is obviously stronger, but our proof shows it can be found through analytic tools at least for large $p$ (which can be made explicit if necessary). The real interest of Corollary \ref{cormain} lies in part $(b)$ and $(c)$, and a possible application (which would require further work) to find a linear lower bound for the dimension of the winding quotient of $J(d,p)$, following the methods of \cite{IwaniecSarnak00}.

$(b)$ The restriction $d \in \{2,3,5,7,13\}$ comes from Theorem \ref{thmmain}. For $d \notin \{2,3,5,7,13\}$ prime, there is actually a natural rank zero quotient of $J'(d,p)$ (not its twist) given by the Eisenstein quotient of $J_0(d)$, because of \eqref{eqdecompjacnonsplit}. The existence of a rank zero quotient of $J(d,p)$ (or its twist $J(d,p,\chi)$) for $d \notin \{2,3,5,7,13\}$ does not seem attainable by the methods used here.

$(c)$ The application to Mazur's method for $\Q$-curves as designed in (\cite{Ellenberg05}, Proposition 3.6) is doable for large enough $p$, $K$ quadratic real and $d \in \{2,3,5,7,13\}$ with $d$ split in $K$ i.e. $\chi(d)=1$, which is the limitation exposed in (\cite{Ellenberg05}, Remark 3.7), because we obtain points $P$ such that $P^\sigma = w_d \cdot P$ (modulo torsion). For $d \notin \{2,3,5,7,13\}$ prime, the techniques do not work here, because of $(b)$ and because it would amount to proving that the jacobian of $X_0(d)/w_d$ has a rank zero quotient, which is not true if we admit the conjecture of Birch and Swinnerton-Dyer because for $f \in S_2(\Gamma_0(d))^+$, $L(f,1) =0$. Notice that this problem is related to the existence of quadratic $\Q$-curves of degree $d$ for general prime $d$, and a still open conjecture of Elkies \cite{Elkies04} states that for large enough $d$, there are none.

\end{rem}

We will begin with the useful notations in section \ref{sectionnotations}, followed by a reminder on the approximate functional equation to estimate the $L(f \otimes \chi,1)$ in section \ref{sectioneqfoncapprox}. We then prove Lemma \ref{a1Lchinewdp2} which allows us to separate the contribution of the newforms in the moment of the $a_1(f) L(f \otimes \chi,1)$ over $S_2(\Gamma_0(dp^2))^{+_{p^2}}$ in section \ref{sectionkeylemma}. In section \ref{sectionPeterssonrestreint}, we obtain the mentioned Petersson trace formula restricted to Atkin-Lehner involution spaces. Finally, we compute all the terms involved in the computation of the moment, leading to the proof of Theorem \ref{thmmain} in section \ref{sectionfinale} and conclude with the proof of Corollary \ref{cormain} there.

\section*{Acknowledgements}
I wish to thank Olga Balkanova for her helpful answers to related problems, and the anonymous referees for their careful reading and very constructive remarks about the previous versions.

\section{Notations}
\label{sectionnotations}

Let $N$ be a positive integer and $\Hcal$ be the Poincaré half-plane.

\begin{itemize}
	\item $S_2(\Gamma_0(N))$ is the complex vector space of cuspidal forms of weight 2 for $\Gamma_0(N)$, and we add the superscripts $\{+,-,\textrm{old},\textrm{new}\}$ to refer respectively to the subspaces made up of the forms $f$ such that $f _{|w_N} = f$, $f _{|w_N} = -f$, $f$ is old, $f$ is new. We will  cumulate the superscripts when it is nonambiguous, for example $S_2 (\Gamma_0(N))^{+,\textrm{old}}$ is the subspace of oldforms $f$ such that $f_{|w_N} = -f$.
	
	\item For $f \in S_2 (\Gamma_0(N))$, one has the $q$-expansion
	\[
		f(z) = \sum_{n \geq 1} a_n(f) e^{ 2 i \pi n z}  \qquad (z \in \Hcal)
	\]
and we will keep this notation $a_n(f)$ throughout. The $L$-function associated to $f$ is defined as a holomorphic series over the domain $\{ \Re(s)>2 \} $ by 
\[
	L(f,s) = \sum_{n \geq 1} \frac{a_n(f)}{n^s}.
\]

\item For $f \in S_2(\Gamma_0(N))$ and $\chi$ a Dirichlet character, the twist $f \otimes \chi$ is defined on $\Hcal$ as the series 
\[
	(f \otimes \chi)(z) = \sum_{n \geq 1} \chi(n) a_n(f) e^{ 2 i \pi n z} \qquad (z \in \Hcal)
\]
and its $L$-function on $\{ \Re(s)>2 \}$ is in the same fashion defined by the holomorphic series
\[
	L (f \otimes \chi,s) = \sum_{n \geq 1} \frac{\chi(n) a_n(f)}{n^s},
\]
which extends to a holomorphic function on $\C$ (Lemma \ref{lemsumLfchi} $(b)$ for details).

\item For every $m \in \N$, define $a_m$ the linear form associating to any modular form $f \in S_2(\Gamma_0(N))$ the coefficient $a_m(f)$, and $L_\chi : f \mt L(f \otimes \chi,1)$

	\item For $\gamma = \matabcd \in \GL_2^+ (\R)$ and $z \in \Hcal$, define 
\[
\gamma \cdot z: = \frac{a z +b}{cz +d}, \qquad j_\gamma(z) := cz + d.
\]

For any holomorphic function $f$ on $\Hcal$, let $f_{|\gamma}$ be the function on $\Hcal$ defined by  
\[
	f_{|\gamma} (z) := \frac{\det \gamma}{j_\gamma(z)^2} f( \gamma \cdot z).
\]
We recall that this defines a right action of $\GL_2^+ (\R)$, satisfying the following formulas
\begin{equation}
\label{formulesbaseactionsurHcal}
\Im (\gamma \cdot z) = \frac{ (\det \gamma) \Im z }{|j_\gamma(z)|^2}, \qquad \gamma \cdot z = \frac{a}{c} - \frac{\det \gamma}{c j_\gamma(z)}, \qquad j_{\gamma \gamma'} (z) = j_\gamma (\gamma' \cdot z) j_{\gamma'} (z). 
\end{equation}
which we will frequently use without specific mention.

\item The Petersson scalar product $\langle \cdot, \cdot \rangle_N$ on $S_2(\Gamma_0(N))$ is defined by 
\[
	\langle f,g \rangle_N = \int_{D} \overline{f(x+iy)} g(x+iy) dx dy,
\]
where $D$ is a choice of fundamental domain for the action of $\Gamma_0(N)$ on $\Hcal$.
Defined as such, the Petersson scalar product  depends on the chosen congruence subgroup, for example when $N'$ divides $N$ and $f,g \in S_2 (\Gamma_0 (N'))$, we have 
\begin{equation}
\label{eqPeterssonchangementniveau}
	\langle f,g \rangle_{N} = [\Gamma_0 (N) : \Gamma_0 (N')] \langle f,g \rangle_{N'}.
\end{equation}

\item For every positive divisor $Q$ of $N$ such that ${\textrm{gcd}}(Q,N/Q) = 1$, choose $W_Q$ a matrix of the following form: 
\begin{equation}
\label{eqdefAtkinLehner}
	W_Q := \begin{pmatrix} Q & y \\ N & Qt \end{pmatrix}, \quad y,t \in \Z, \quad \det W_Q = Q.
\end{equation}
For every $f \in S_2 (\Gamma_0(N))$, the function $f_{|W_Q}$ does not depend on the choice of $W_Q$, and the \textit{Atkin-Lehner involution of degree $Q$ on $S_2(\Gamma_0(N))$} is the corresponding involution on this space (noted $w_Q$ to emphasize its canonical nature). For $\varepsilon = \pm 1$, the space $S_2 (\Gamma_0 (N))^{\varepsilon_Q}$ is the subspace of $S_2(\Gamma_0(N))$ made up with the modular forms $f$ such that $f_{|w_Q} = \varepsilon f$, for example $S_2 (\Gamma_0(N))^{+} = S_2 (\Gamma_0(N))^{+_N}$. Note that the definition of $W_Q$ generally depends on $N$, so we will precise (unless the context is obvious) on which spaces we are considering them. For more details about these involutions, see \cite{AtkinLehner70}. In particular, notice that for $f \in S_2(\Gamma_0(N))$, 
\begin{equation}
\label{eqexplicitefwN}
	f_{|w_N} (z) = \frac{1}{N z^2} f \left( \frac{i}{Nz} \right).
\end{equation}

\item For any subspace $V$ of $S_2(\Gamma_0(N))$ and any linear forms $A,B$ on $V$, define 
\begin{equation}
	(A,B)_N^V : = \sum_{f \in \Fcal_V} \overline{A(f)}B(f),
\end{equation}
where $\Fcal_V$ is a Petersson-orthonormal basis of $V$. This defines a scalar product on $V^*$ independent of the choice of $\Fcal_V$. We will in particular denote by $(A,B)_N$ the scalar product of $A$ and $B$ on the whole space $S_2 (\Gamma_0(N))$, again adding natural superscripts corresponding to how $V$ is defined as a subspace of $S_2(\Gamma_0(N))$. For example, Theorem \ref{thmmain} is exactly reformulated as 
\begin{equation}
\label{eqthmmainreformule}
(a_1,L_\chi)_{dp^2}^{+_{p^2}, \textrm{new}} = 2 \pi + O \left( \frac{\sqrt{D}(\log(D)+1)^3 \log(p)^2}{p^2} \right).
\end{equation}

\end{itemize}

\section{The approximate functional equation}
\label{sectioneqfoncapprox}

We will recall here some necessary results to provide a way of evaluating $L(f \otimes \chi,1)$.

\begin{lem}
\label{lemsumLfchi}
	Let $\chi$ be a quadratic character of conductor $D$ and $f \in S_2(\Gamma_0(N))$ with $N$ prime to $D$.
	
	$(a)$ The twisted modular form $f \otimes \chi$ belongs to $S_2(\Gamma_0(D^2 N))$ and 
	\begin{equation}
	\label{eqtwistf}
		(f \otimes \chi)_{|w_{D^2N}} = \chi(-N) f_{|w_N} \otimes \chi.
	\end{equation}

$(b)$ 	The holomorphic series $L(f \otimes \chi, \cdot)$ extends to an holomorphic function on $\C$ and for every $x>0$, one has 

	\begin{equation}
	\label{eqapproxLfchi}
	L(f \otimes \chi,1) = \sum_{n=1}^{+ \infty} \frac{\chi(n) a_n(f)}{n} e^{- \frac{2 \pi n}{x}} - \chi(-N) \sum_{n=1}^{+ \infty} \frac{\chi(n)  a_n(f_{|w_N})}{n} e^{- \frac{2 \pi n x}{D^2 N}}.
	\end{equation}

$(c)$ In particular, if $f_{|w_N} = \chi(-N) \cdot f$, we have $L(f \otimes \chi,1) = 0$.
\end{lem}

\begin{proof}
	\hspace{\fill}
	
	$(a)$ This is a classical result, which can for example be found in (\cite{Bump}, $\mathsection$ $I.5$).

	 $(b)$ Let $M=D^2 N$. We will reprove below that for every $g \in S_2(\Gamma_0(M))$, the $L$-function of $g$ extends to $\C$ and 
\begin{equation}
\label{eqintermL}
	L(g,1) = \sum_{n=1}^{+ \infty} \frac{a_n(g)}{n} e^{- \frac{2 \pi n}{x}} - \sum_{n=1}^{+ \infty} \frac{a_n(g{|w_{M}})}{n} e^{- \frac{2 \pi n x}{M}},
\end{equation}
so that $(b)$ is a direct consequence of $(a)$ and \eqref{eqintermL}.

On $\Re(s)>2$, let us define the completed $L$-function of $g$ by 
\[
\Lambda(g,s) : = \frac{M^{s/2} \Gamma(s)}{(2\pi)^s} L(g,s).
\]
As usual, by absolute convergence, we can write 
\begin{eqnarray*}
	\Lambda(g,s) & = & M^{s/2} \sum_{n=1}^{+ \infty} a_n(g) \int_0^{+ \infty} e^{-t} \left( \frac{t}{2 \pi n} \right)^s \frac{dt}{t} \\
	& = &  M^{s/2} \sum_{n=1}^{+ \infty} a_n(g)  \int_0^{+ \infty} e^{ - 2 \pi n y} y^s \frac{dy}{y} \\ 
	& = & M^{s/2} \int_0^{+ \infty} \sum_{n=1}^{+ \infty} a_n(g) e^{ - 2 \pi n y} y^s \frac{dy}{y} \\
	& = & \int_0^{+ \infty} g(iy) (M^{1/2} y)^s \frac{dy}{y}.
\end{eqnarray*}
We choose $x>0$ and split the integral between $[1/x,+ \infty[$ and $]0,1/x]$. We obtain 
\begin{eqnarray*}
\Lambda(g,s) & = & \int_{1/x}^{+ \infty} g(iy) (M^{1/2} y)^s \frac{dy}{y} + \int_{0}^{1/x} g(iy) (M^{1/2} y)^s \frac{dy}{y} \\
& = & \int_{1/x}^{+ \infty} g(iy) (M^{1/2} y)^s \frac{dy}{y} + \int_{x/M}^{+ \infty} g \left( \frac{i}{Mt} \right) \left( \frac{1}{M^{1/2}t} \right)^s \frac{dt}{t} \quad \left(t = \frac{1}{My} \right) \\
& = & \int_{1/x}^{+ \infty} g(iy) (M^{1/2} y)^s \frac{dy}{y} + \int_{x/M}^{+ \infty} M (it)^2 g_{|w_M} (it) \left( \frac{1}{M^{1/2}t} \right)^s \frac{dt}{t}
\end{eqnarray*}
using \eqref{eqexplicitefwN}. We obtain the integral expression 
\[
	\Lambda(g,s) = M^{s/2} \int_{1/x}^{+ \infty} g(iy) y^s \frac{dy}{y} - M^{1 - s/2} \int_{x/M}^{+ \infty}  g_{|w_M} (iy) y^{2-s} \frac{dy}{y}.
\]
This immediately proves that $\Lambda(g)$ extends to an entire function satisfying the functional equation 
\begin{equation}
\label{eqfoncLambdaaux}
\Lambda(g,2-s) = - \Lambda(g_{|w_M}, s),
\end{equation}
hence $L(g,\cdot)$ extends to an entire function on $\C$.
For the central value $s=1$, we have 
\begin{eqnarray*}
	\Lambda(g,1) & = & \sqrt{M} \sum_{n=1} ^{+ \infty} a_n (g) \int_{1/x}^{+ \infty} e^{ - 2 \pi n y} dy - \sqrt{M} \sum_{n=1}^{+ \infty} a_n(g_{|w_M}) \int_{x/M}^{+ \infty} e^{ - 2 \pi n y}  dy \\
	& = & \sqrt{M}\left( \sum_{n=1} ^{+ \infty} \frac{a_n (g)}{2 \pi n} e^{ - \frac{2 \pi n}{x}} -  \sum_{n=1} ^{+ \infty} \frac{a_n (g_{|w_M})}{2 \pi n} e^{ - \frac{2 \pi nx}{M}} \right),
\end{eqnarray*}
which proves \eqref{eqintermL} as $L(g,1) = 2 \pi \Lambda(f,1) / \sqrt{M}$. 

$(c)$ This is a straightforward consequence of $(b)$ by applying \eqref{eqapproxLfchi} to $x = D \sqrt{N}$, for which the two integrals on the right cancel each other.
\end{proof}

\section{A key lemma to isolate the contribution of the newforms}
\label{sectionkeylemma}

\begin{lem}
\label{a1Lchinewdp2}
Let $d \in \{2,3,5,7,13\}$ and $\chi$ be an even Dirichlet character with conductor $D$ prime to $d$. For every prime number $p$ not dividing $dD$, we have
\begin{equation}
\label{eqa1Lchinewdp2}
(a_1,L_\chi)_{dp^2}^{+_{p^2}, {\textrm{new}}} = (a_1,L_\chi)_{dp^2}^{+_{p^2}} - \frac{1}{p-1}(a_1,L_\chi)_{dp}^{\chi(p)_{p}}.
\end{equation}

\end{lem}

%
%
%
%

\begin{rem}
\label{remmarchepasautrescas}
	We made here an assumption on $d$ and a choice of eigenvalue for $w_{p^2}$. Let us discuss the (similar) reasons behind these choices.
	
$\bullet$ For the choice of sign $-_{p^2}$, either the sign of eigenvalue for $w_d$ is $-\chi(d)$, then the sign of the twisted $L$-function is $(-1)$, giving an automatic vanishing (Lemma \ref{lemsumLfchi} $(c)$), either it is $\chi(d)$, and then the proof below does not work : indeed, we could not evaluate exactly the contribution of the $d$-old space, as the formula analogous to \eqref{eqfA1gAp} for $f \in S_2(\Gamma_0(p^2)), g \in S_2(\Gamma_0(p^2))$ is 
\[
	\langle f_{|A_1}, g_{|A_d} \rangle_{dp^2} = \langle f_{|T_d}, g \rangle_{p^2},
\]
and the eigenvalues of $T_d$ on $S_2(\Gamma_0(p^2))$ are not simply $\pm 1$. Actually, as one knows that the eigenvalues of $T_d$ are of modulus bounded by $2 \sqrt{d}$ and that every $a_1(f) L_\chi(f)$ is nonnegative when $f$ is an eigenform (see \cite{Guo96}), one can easily compute that the contribution of the $d$-old forms  is bounded in absolute value by a term of the shape $O(\frac{1}{d}) (a_1,L_\chi)_{p^2}^{-}$. It is not needed in the present case, therefore we do not give more details.

$\bullet$ The number $d$ is assumed in $\{2,3,5,7,13\}$ to ensure we can evaluate the contribution of the $p$-old space in $S_2(\Gamma_0(dp))$, which is automatically trivial in this case. If $d$ is a larger prime, as the analogue of the formula \eqref{eqfA1gAp} for $f \in S_2(\Gamma_0(d), g \in S_2(\Gamma_0(d)$ is
\[
	\langle f_{|A_1}, g_{|A_p} \rangle_{dp} = \langle f_{|T_p}, g \rangle_d,
\]
we cannot obtain an exact formula such as \eqref{eqa1Lchinewdp2} for the same reason as above, but we could again bound the contribution of these old forms by some term of the shape $O(\frac{1}{p}) (a_1,L_\chi)_d$.

\end{rem}

\begin{proof}
	By definition of the newforms and oldforms, we have the orthogonal decomposition 
	\[
	S_2(\Gamma_0 (dp^2))^{+_{p^2}} = S_2(\Gamma_0 (dp^2))^{+_{p^2}, \textrm{new}} \oplus S_2(\Gamma_0 (dp^2))^{+_{p^2},\textrm{old}}	
	\]
hence 
\begin{equation}
\label{eqinterma1Lchi}
	(a_1,L_\chi)_{dp^2}^{+_{p^2}, \textrm{new}} = (a_1,L_\chi)_{dp^2}^{+_{p^2}} - (a_1,L_\chi)_{dp^2}^{+_{p^2}, \textrm{old}},
\end{equation}
so we have to relate this scalar product on the oldpart to the right term in the Lemma.

Following the notations of \cite{AtkinLehner70}, let us define, for every positive integer $\delta$, $A_\delta = \begin{pmatrix} \delta & 0 \\ 0 & 1 \end{pmatrix}$ and for every positive integers $M$ and $N$ such that $M|N$ and $\delta$ divides $N/M$, the operator ${A_\delta : S_2 (\Gamma_0(M)) \ra S_2 (\Gamma_0(N))}$
	\[
		f \mt f_{|A_\delta}.
	\]
Looking at the $q$-expansions, we immediately see that for every $\delta \geq 1$, $\left( f_{|A_\delta}\right) \otimes \chi = \chi (\delta) \left( f \otimes \chi \right)_{|A_\delta}$ hence
\begin{equation}
\label{eqLchiAdelta}
	L_\chi (f_{|A_\delta}) = \int_0^{+ \infty} (f_{|A_\delta} \otimes \chi) (iu) du = \delta \cdot \chi (\delta) \int_0^{+ \infty} (f \otimes \chi) (i \delta u) du = \chi(\delta) L_\chi(f).
\end{equation}
In particular, 
\begin{equation}
\label{eqLchidegen}
L_\chi (f_{|A_1}) = L_\chi \textrm{ and } L_\chi (f_{|A_\delta}) = 0 \textrm{ if } L_\chi (f) = 0.	
\end{equation}

By definition, the old part of $S_2(\Gamma_0 (dp^2))$ is the subspace spanned by the $f_{|A_1}, f_{|A_p}$ with $f \in S_2(\Gamma_0 (dp))$ (it is the $p$-old space) and by the $f_{|A_1}, f_{|A_d}, f \in S_2 (\Gamma_0 (p^2))$ (it is the $d$-old space). Let us begin with the $d$-old space: as $d$ and $p$ are coprime, by Lemma 26 of \cite{AtkinLehner70}, for all $f \in S_2 (\Gamma_0 (p^2))$, 
\begin{equation}
	\left( f_{|A_1} \right) _{|w_{p^2}} = \left(f_{|w_{p^2}} \right)_{|A_1}\quad {\textrm{and}} \quad \left( f_{|A_d} \right) _{|w_{p^2}} = \left(f_{|w_{p^2}} \right)_{|A_d}.
\end{equation}
In particular, $f_{|A_1}$ and $f_{|A_d}$ have the same eigenvalue for $w_{p^2}$ in $S_2(\Gamma_0(dp^2))$ as $f$ in $S_2(\Gamma_0(p^2))$, which proves that $S_2 (\Gamma_0 (dp^2))^{+_{p^2}, d-{\textrm{old}}}$ is generated by the $f_{|A_1}, f_{|A_d}$ where $f \in S_2 (\Gamma_0 (p^2))^+$. The Lemma \ref{lemsumLfchi} $(c)$ tells us in this case that $L_{\chi}(f) = 0$ because $\chi(-p^2) = 1$, hence $L_\chi$ is zero on the $d$-old space by \eqref{eqLchidegen} and 
\[
(a_1,L_\chi)_{dp^2}^{+_{p^2}, \textrm{old}} = (a_1,L_\chi)_{dp^2}^{+_{p^2},p- \textrm{old}}.
\]
We will now compute the contribution of the $p$-old space. Our hypothesis on $d$ ensures that $S_2(\Gamma_0 (dp)) = S_2 (\Gamma_0 (dp))^{p-\textrm{new}}$ because $S_2(\Gamma_0(d))=0$. Let $f$ and $g$ be two ($p$-new) eigenforms on $S_2 (\Gamma_0 (dp))$. By definition of the Petersson scalar product, we immediately obtain 
\begin{equation}
\label{produitscalaire1}
	\langle f_{|A_1}, g_{|A_1} \rangle_{dp^2} = [\Gamma_0 (dp) : \Gamma_0(dp^2)] \langle f,g \rangle_{dp} = p \langle f,g \rangle_{dp},
\end{equation}
and 
\[
\langle f_{|A_p},g_{|A_p} \rangle_{dp^2} = p^2 \int_{\Dcal} \overline{f(px + ipy)} g(px + ipy) dx dy = \int_{p \Dcal} \overline{f} (x + i y) g(x+iy) dx dy,
\]
where $\Dcal$ is a fundamental domain for $\Gamma_0 (dp^2)$. It readily implies that $p \Dcal$ is a fundamental domain for the subgroup $\Gamma$ of matrices of $\Gamma_0(d)$ which are diagonal modulo $p$, and this subgroup is of index $p$ in $\Gamma_0(dp)$ so we obtain 
\begin{equation}
\label{produitscalaire2}
\langle f_{|A_p},g_{|A_p} \rangle_{dp^2} = p \langle f,g \rangle_{dp}.
\end{equation}
Next, using again a linear change of variables, we obtain 
\begin{eqnarray*}
	\langle f_{|A_1}, g_{|A_p} \rangle_{dp^2} & = & p \int_{\Dcal} \overline{f (x+ i y)} g(p(x+iy)) dx dy \\
	& = & \frac{1}{p} \int_{p \Dcal} \overline{f ((x+ i y)/p)} g(x+iy) dx dy \\
	& = & \langle f_{|A_p^{-1}}, g \rangle_{\Gamma}
\end{eqnarray*}
with the same $\Gamma$ as above, but with Lemma 12 and notations (2.2) and (3.1) of \cite{AtkinLehner70}, as $\begin{pmatrix} 1 & j \\ 0 & 1 \end{pmatrix}$ is a system of coset  representatuives of $\Gamma_0(dp) \backslash \Gamma$ for $0 \leq j \leq p-1$ and $f$ and $f_{|A_p^{-1}}$ are both modular forms for $\Gamma$, we get
\begin{equation}
\label{eqfA1gAp}
\langle f_{|A_1}, g_{|A_p} \rangle_{dp^2} = \langle f_{|U_p}, g \rangle_{dp}
\end{equation}
and as $f$ is a $p$-new eigenform, it is an eigenform for $U_p$ and $w_p$ (Theorem 3 of \cite{AtkinLehner70}) and the eigenvalues are opposite. Defining $\varepsilon_f \in \{ \pm 1 \}$ such that 
\begin{equation}
\label{eqdefepsilonf}
f_{|w_p} = \varepsilon_f \cdot f,	
\end{equation}
we finally obtain 
\begin{equation}
\label{produitscalaire3}
\langle f_{|A_1}, g_{|A_p} \rangle_{dp^2} = - \varepsilon_f \langle f,g \rangle_{dp}.
\end{equation}

Now, let $\Bcal$ be an orthonormal eigenbasis of $S_2 (\Gamma_0(dp)) = S_2(\Gamma_0(dp))^{p-\textrm{new}}$ (for the Hecke operators $T_q, q \neq d,p$ and $U_p$). When $f$ runs through $\Bcal$, the vector spaces $\Vect (f_{|A_1}, f_{|A_p})$ are pairwise orthogonal because of the formulas \eqref{produitscalaire1}, \eqref{produitscalaire2} and \eqref{produitscalaire3}. This allows us to build from $\Bcal$ an orthonormal basis of $S_2(\Gamma_0(dp^2))^{+_{p^2},p-{\textrm{old}}}$ in the following way.
%
From Lemma 26 of \cite{AtkinLehner70}, we know that for $f \in \Bcal$, with the notation \eqref{eqdefepsilonf},
\[
	(f_{|w_p})_{|A_p} = (f_{|A_1})_{|w_{p^2}} \quad {\textrm{and}} \quad (f_{|A_p})_{|w_{p^2}} = (f_{|w_p})_{|A_1} = \varepsilon_f f_{|A_1}.
\]
An orthogonal basis of $S_2(\Gamma_0(dp^2))^{+_{p^2},p-{\textrm{old}}}$ is then made up with the
\[
	f_{|A_1} + \left( f_{|A_1} \right)_{|w_{p^2}} =  f_{|A_1} + \varepsilon_f f_{|A_p} , \quad f \in \Bcal.
\]
For $f \in \Bcal$, we know from the formulas \eqref{produitscalaire1}, \eqref{produitscalaire2} and \eqref{produitscalaire3} that: 
\[
	\langle f_{|A_1} + \varepsilon_f f_{|A_p}, f_{|A_1} + \varepsilon_f f_{|A_p} \rangle_{dp^2} = (2 p - 2 \varepsilon_f^2) \langle f,f \rangle_{dp} = 2 (p-1). 
\]
To summarize, an orthonormal basis of $S_2(\Gamma_0 (dp^2))^{+_{p^2}, p-\textrm{old}}$ is obtained by taking the elements $f$ of $\Bcal$ and considering the $(f_{|A_1} + \varepsilon_f f_{|A_p})/\sqrt{2(p-1)}$. Finally, by \eqref{eqLchidegen},
\[
	\overline{a_1} (f_{|A_1} + \varepsilon_f f_{|A_p}) L_\chi (f_{|A_1} + \varepsilon_f f_{|A_p}) = \overline{a_1(f)} (L_\chi (f) + \varepsilon_f \chi(p) L_\chi(f)),
\]
in particular the left term is zero when $\varepsilon_f = - \chi(p)$. Summing this over all $f \in \Bcal$ such that $\varepsilon_f = \chi(p)$ and after orthornormalisation, we get 
\[
	(a_1,L_\chi)_{dp^2}^{+_{p^2},p-{\textrm{old}}} = \frac{1}{p-1} (a_1,L_\chi)_{dp}^{\chi(p)_p},
\]
which proves the Lemma.
\end{proof}

We now need to calculate both terms on the right of \eqref{eqa1Lchinewdp2}, and to do this we will use a version of Petersson trace formula in the next section.

\section{The semi-orthogonality relation with respect to Atkin-Lehner involutions}
\label{sectionPeterssonrestreint}

Let us begin with the necessary definitions for the trace formulas.

\begin{defi}[Kloosterman sums and Bessel function]
\label{KloostermanBessel}
\hspace*{\fill} 

For all positive integers $m,n,c$, the \textit{Kloosterman sum associated to } $m,n,c$ is defined by 
\[
	S(m,n;c) = \sum_{k \in (\Z /c \Z)^*} e^{ 2 i \pi (m k + n k^{-1})/c}.
\]
The Kloosterman sums satisfy the Weil bounds (\cite{IwaniecKowalski}, Corollary 11.12): 
\begin{equation}
\label{bornesdeWeilKloosterman}
|S (m,n ;c) |\leq (m,n,c) ^{1/2} \tau(c) \sqrt{c},
\end{equation}
with $(m,n,c)$ the gcd of $m$,$n$ and $c$ and $\tau(c)$ the number of positive divisors of $c$. 
%

The \textit{Bessel function of the first kind and order 1} is the entire function $J_1$ defined by  
\[
J_1 (z) = \sum_{n=0}^{+ \infty} \frac{(-1)^n}{n ! (n+1) !} \left( \frac{z}{2} \right)^{2 n +1}.
\]
It has the following integral representation (\cite{WatsonBessel}, 6.21, Formula 8)
\begin{equation}
\label{representationintegraleJ1}
	J_1 (z) = \frac{z}{4 i \pi} \int_{x - i\infty}^{x + i \infty} \frac{e^{w - \frac{z^2}{4w}}} {w^2} dw 
\end{equation}
for all $z \in \C$ and all $x>0$.
\end{defi}

The goal of this subsection is to prove Propositions \ref{propformulesdestraces} and \ref{propformulesdestraceslaplusgenerale}. With our notations, the left-hand term of \eqref{formuledestracesgeneralisee} is exactly 
\[
\frac{1}{2 \pi \sqrt{mn}}(a_m,a_n)_N^{\varepsilon_Q}.
\]

Before the proof, let us make some remarks about Proposition \ref{propformulesdestraces}.

\begin{rem}
\hspace*{\fill}
\label{remformuletraces}

$(a)$ Summing for any $Q$ the formulas for $\varepsilon = 1$ and $\varepsilon = -1$, we recover the original Petersson trace formula (\cite{IwaniecKowalski}, Proposition 14.5), which generalises to every weight $k \geq 2$. However, its proof for $k=2$ is more involved because the Poincaré series cannot be defined as uniformly convergent series so we will focus on this case (it is also the only one we need), but it is very likely to be generalised to $k>2$ as well. The trace formula above has been originally proved for $Q=N$ in the Chapter 3 of \cite{Akbary}, but to our knowledge, not for any other $Q$.

$(b)$ For $Q=N$ prime, formula \eqref{formuledestracesgeneralisee} can be found (in a different form) in (\cite{IwaniecKowalski}, Proposition 14.25). Notice that there is a mistake in one of the arguments of $J_1$ in that book, as it should actually be rewritten (with our notations) for prime level $q$ and $m,n \geq 1$ : 
\begin{eqnarray}
\label{eqPeterssonrestreintversionIK}
	\frac{(a_m,a_n)_q^{\varepsilon}}{2 \pi \sqrt{mn}}  = \delta_{mn} - 2 \pi \sqrt{q} \delta_{m,nq} & - & 2 \pi \sum_{q|c} \frac{S(m,n;c)}{c} J_1 \left( \frac{4 \pi \sqrt{mn}}{c} \right) \\
	& + & 2 \pi \varepsilon \sqrt{q} \sum_{q|c} \frac{S(m,nq;c)}{c} J_1 \left(\frac{4 \pi \sqrt{mnq}}{c}  \right) \nonumber.
\end{eqnarray}
The proof of this result is simply based on the natural system of formulas (which crucially uses Theorem 3 of \cite{AtkinLehner70} hence the hypothesis $q$ prime)
\begin{eqnarray*}
	(a_m,a_n)_q & = & (a_m,a_n)_q^{+} + (a_m,a_n)_q^{-} \\
	(a_m,a_{nq})_q & = & - (a_m, a_n)_q^+ + (a_m,a_n)_q^-
\end{eqnarray*}
combined with the original Petersson trace formula. Notice there is a $\delta_{m,nq}$ appearing here but not in \eqref{formuledestracesgeneralisee} so even under this form, the fact that formulas \eqref{formuledestracesgeneralisee} and \eqref{eqPeterssonrestreintversionIK} are the same is not obvious. To see it, we use that for $c \geq 1$ such that $q || c$, 
\[
	S(m,nq;c) = - S(m, n q^{-1}; c/q)
\]
(e.g. by Theorem 68 of \cite{HardyWright}), and that for all $m,n \geq 1$, we can check (separating between oldforms and newforms) that
\[
0 = (a_m,a_{nq})_{q^2} = \delta_{m,nq} - 2 \pi \sum_{q^2 | c} \frac{S(m,nq;c)}{c} J_1 \left( \frac{4 \pi \sqrt{mnq}}{c} \right),
\]
and we readily obtain the equivalence using these two results.
\end{rem}

To prove \eqref{formuledestracesgeneralisee}, we will use the Poincaré series in weight 2, whose classical properties are recalled below (\cite{IwaniecKowalski}, Lemma 14.2 and \cite{Rankin}, section 5.7).

\begin{defi}[Poincaré series of weight 2]
\hspace*{\fill}

	For every positive integers $n,N$ there are cuspidal forms of weight 2 for $\Gamma_0(N)$ denoted by $P_n (\cdot,N)$ and called  \textit{Poincaré series of weight 2} such that
	
	$(a)$ The Poincaré series $P_n(\cdot, N)$ is the uniform limit on every compact subset of $\Hcal$ when $s \ra 0^+$ of the series $P_n(\cdot,s,N)$ defined by 
	\[
		P_n(z,s,N) : = \sum_{\gamma \in \Gamma_\infty \backslash \Gamma_0 (N)} \frac{e^{ 2 i \pi n \gamma \cdot z}}{j_\gamma(z)^2 |j_\gamma(z)|^{2s}}.
	\]
	These series satisfy by uniform convergence the transformation formula
	\[
		P_n (\gamma \cdot z,s,N) = j_\gamma (z)^2 |j_\gamma(z)|^{2s} P_n(z,s,N).
	\]
	$(b)$ For every  $m>0$, 
	\begin{equation}
	a_m (P_n (\cdot,N)) =  \delta_{mn}- 2 \pi \left( \sqrt{\frac{m}{n}} \sum_{\substack{c >0 \\ N |c}}  \frac{ S(m,n;c) }{c} J_1 \left( \frac{4 \pi \sqrt{mn}}{c} \right) \right).
	\end{equation}
	
	$(c)$ For every $f \in S_2 (\Gamma_0(N))$, 
	\[
	\langle f,P_n (\cdot,N) \rangle_N  = \frac{\overline{a_n(f)}}{4 \pi n}.
	\]
\end{defi}

%
%
%
%
%

The essential result to prove for our trace formula is the following. 

\begin{prop}
	For every positive integers $m,n,N$ and every divisor $Q>1$ of $N$ such that $(Q,N/Q)= 1$, 
	\begin{equation}
	\label{eqFouriercoeffPnwQ}
	a_m(P_n (\cdot,N)_{|w_Q}) = 	
 	- 2 \pi \sqrt{\frac{m}{n}} \sum_{m \geq 1} \sum_{\substack{c>0 \\ (N/Q) |c \\(Q,c) = 1}} \frac{S(m,nQ^{-1} ;c)}{c \sqrt{Q}}J_1 \left( \frac{4 \pi \sqrt{mn}}{c\sqrt{Q}} \right).
 	\end{equation}
\end{prop}

Let us explain first why this implies the trace formula. Define 
\[
	P_n^{+_Q} (\cdot,N) := P_n (\cdot,N) + P_n (\cdot,N)_{|w_Q}.
\]
It belongs to $S_2(\Gamma_0 (N))^{+_Q}$, and for any $f \in S_2(\Gamma_0 (N))^{+_Q}$, as $w_Q$ is self-adjoint, we have 
\[
	\langle f, P_n^{+_Q} (\cdot,N) \rangle = \langle f,P_n (\cdot,N) \rangle + \langle f, P_n (\cdot,N)_{|w_Q} \rangle =   \langle f,P_n (\cdot,N) \rangle + \langle f_{|w_Q}, P_n(\cdot,N) \rangle = 2 \langle f,P_n (\cdot,N) \rangle.
\]
Hence, for $\Fcal_{N,Q}$ an orthonormal basis of $S_2(\Gamma_0 (N))^{+_Q}$, the property $(c)$ of Poincaré series gives us 
\[
P_n^{+_Q} (\cdot,N) = 
	 \sum_{f \in \Fcal_{N,Q}} \langle f,P_n^ {+_Q} \rangle f = 2 \sum_{f \in \Fcal_{N,Q}} \langle f,P_n \rangle f = \sum_{f \in \Fcal_{N,Q}} \frac{\overline{a_n(f)}}{2 \pi n} f ,
\]
and the property $(b)$ of Poincaré series together with Proposition \ref{propformulesdestraces} give us, by identification of Fourier coefficients, for every $m >0$: 
\begin{eqnarray*}
	\frac{\left(a_n,a_m\right)_N^{+_Q}}{2 \pi n}  = \delta_{mn}& - & 2 \pi  \sqrt{\frac{m}{n}} \sum_{\substack{c >0 \\ N |c}}  \frac{ S(m,n;c) }{c} J_1 \left( \frac{4 \pi \sqrt{mn}}{c} \right) \\ & - & 2 \pi \sqrt{\frac{m}{n}} \sum_{m \geq 1} \sum_{\substack{c>0 \\ (N/Q) |c \\(Q,c) = 1}} \frac{S(m,nQ^{-1} ;c)}{c \sqrt{Q}}J_1 \left( \frac{4 \pi \sqrt{mn}}{c\sqrt{Q}} \right),
\end{eqnarray*}
hence the trace formula for $Q$ and $\varepsilon=1$. We can obtain the trace formula for $Q$ and $\varepsilon=-1$ by the same means or by difference with the usual Petersson trace formula, as mentioned earlier.

\begin{rem}
\label{remformuletracesPeterssonrestreintecombinee}
Actually, the same argument gives us the combined Petersson trace formula (which will be useful in section \ref{sectionfinale}), written below.
\end{rem}

\begin{prop}[Restricted Petersson trace formula with multiple eigenvalues]
\label{propformulesdestraceslaplusgenerale}
\hspace*{\fill}

Let $m,n,N$ be three fixed positive integers. Let $E$ be a group morphism from a subgroup $H$ of the group $W$ of Atkin-Lehner involutions on $N$ (identified as the set of $Q |N$ such that $(Q,N/Q)=1$ below) to $\{pm 1 \}$. For every $Q \in W$, let us define 
\[
	S_Q =  2 \pi \sqrt{\frac{m}{n}} \sum_{m \geq 1} \sum_{\substack{c>0 \\ (N/Q) |c \\(Q,c) = 1}} \frac{S(m,nQ^{-1} ;c)}{c \sqrt{Q}}J_1 \left( \frac{4 \pi \sqrt{mn}}{c\sqrt{Q}} \right).
\]
and for $\Bcal$ an eigenbasis of $S_2(\Gamma_0(N))$,
\[
	(a_m,a_n)_N^E := \sum_{\substack{f \in \Bcal \\ \forall Q \in H,\\ f_{|w_Q} = E(Q) f }} \frac{\overline{a_m(f)}a_n(f)}{\|f\|^2}.
\]

Then, we have 
\[
	\frac{|E|}{4 \pi \sqrt{mn}} (a_m,a_n)_N^{E} := \delta_{mn} - \sum_{Q \in H} E(Q) S_Q 
\]
\end{prop}

\begin{proof}
	Let us define
	\[
		P_n(\cdot,N)^E := \sum_{Q \in H} E(Q) P_n(\cdot,N)_{|w_Q}
	\]
By construction, for every $Q \in H$, one has $P_n(\cdot,N)^E_{|w_Q} = E(Q) P_n(\cdot,N)^E$. Now, let $\Bcal^E$ be the subset of $\Bcal$ made up with the eigenforms $f$ having the good signs for the morphism $E$. For every $f \in \Bcal^E$ :
\[
	\langle f,P_n(\cdot,N)^E \rangle = \sum_{Q \in H} E(Q) \langle f,P_n (\cdot,N)_{|w_Q} \rangle  = \sum_{Q \in H} E(Q) \langle f_{|w_Q}, P_n(\cdot,N) \rangle = |E| \langle f,P_n(\cdot,N) \rangle,
\]
hence by the properties of Poincaré series, 
\[
	\langle f,P_n(\cdot,N)^E \rangle = \frac{|E| \overline{a_n(f)}}{4 \pi n}.
\]

Now, $\Bcal^E$ is an orthogonal basis of $S_2(\Gamma_0(N))^E$, and decomposing $P_n(\cdot,N)^E$ on the basis $\Bcal^E$ gives Proposition \ref{propformulesdestraceslaplusgenerale} by identification of $m$-th Fourier coefficients on both sides, using \eqref{eqFouriercoeffPnwQ}.
\end{proof}

Let us now prove the Proposition on Poincaré series.

\begin{proof}
	Let us choose an Atkin-Lehner involution matrix $W_Q = \begin{pmatrix} Q & y \\ N & Q t \end{pmatrix}$ with $y,t \in \Z$ and $Qt - (N/Q)y = 1$. We will compute the Fourier coefficients of $P_n(\cdot,s,N)_{|W_Q}$: here, this depends on the choice of $W_Q$ because $P_n(\cdot,s,N)$ is \textit{not} a modular form.
	
	For any $s>0$, 
	\begin{eqnarray*}
	P_n (\cdot,s,N)_{|W_Q} (z) & = & \frac{\det W_Q}{j_{W_Q} (z)^2} P_n (W_Q \cdot z, s,N)  \\
	& = & \frac{Q}{j_{W_Q} (z)^2} \sum_{\gamma \in \Gamma_\infty \backslash \Gamma_0(N)} \frac{e^{2 i \pi n \gamma W_Q z}}{j_\gamma (W_Q z)^2 |j_\gamma (W_Q z)|^{2s}} \\
	& = & Q |j_{W_Q} (z)|^{2s} \sum_{\gamma \in \Gamma_\infty \backslash \Gamma_0(N) W_Q} \frac{e^{2 i \pi n \gamma z}}{j_\gamma (z)^2 |j_\gamma(z)|^{2s}}.
\end{eqnarray*}
 Now, for every $\matabcd \in \Gamma_0(N)$,
 \[
 	 \matabcd W_Q = \begin{pmatrix} aQ + bN & a y + bQ t \\ c Q + d N & c y + d Q t \end{pmatrix},
 \]
so it belongs to the set of matrices $\begin{pmatrix} a' & b' \\ c' & d' \end{pmatrix}$ with integer coefficients such that $N$ divides $c'$, $Q$ divides $a'$ and $d'$, and with determinant $Q$. Actually, this set is exactly $\Gamma_0(N) W_Q$ as we check immediately by multiplication by $W_Q^{-1}$, and for $Q>1$, $c'$ is necessarily nonzero, so $\Gamma_\infty \backslash \Gamma_0(N) W_Q$ is in natural bijection with the set $\Rcal_{N,Q}$ of triples $(a,c,d)$ of integers such that $c>0, N|c$, $Q|(a,d)$, $ad = Q \mod c$ and $0 \leq a < c$. Moreover, for $\gamma = \matabcd \in \Gamma_0(N) W_Q$ built from such a triple $(a,c,d)$, 
\[
	\gamma \cdot z  = \frac{a}{c} - \frac{Q}{c (cz+d)}
\]
hence
\begin{eqnarray*}
	P_n (\cdot,s,N)_{|W_Q} (z) & = & Q |j_{W_Q} (z)|^{2s} \sum_{(a,c,d) \in \Rcal_{N,Q}} \frac{e^{ 2 i \pi n a/c} e^{- 2 i \pi \frac{nQ}{c(cz+d)}}}{(cz+d)^2 |cz+d|^{2s}} \\
	& = & Q |j_{W_Q} (z)|^{2s} \sum_{\substack{c>0 \\ N |c \\ (Q,c/Q) = 1}} \frac{1}{c^{2 + 2s}} \sum_{\substack{a \\ 0 \leq a < c \\ Q |a}} e^{ 2 i \pi n a/c} \sum_{\substack{d \\ Q|d \\ ad \equiv Q [c]}} \frac{e^{- \frac{2 i \pi n Q}{c^2 (z+ d/c)}} }{(z+d/c)^2 |z+d/c|^{2s}}.
\end{eqnarray*}
For fixed $a$ and $c$, the set of $d$ satisfying the property in the second sum is a congruence class modulo $c$, so we can choose its representative $d'$ between 0 and $c-1$, therefore  
\[
	\sum_{\substack{d \\ Q|d \\ ad \equiv Q [c]}} \frac{e^{- \frac{2 i \pi nQ}{c^2 (z+ d/c)}} }{(z+d/c)^2 |z+d/c|^{2s}} = \sum_{\ell \in \Z} \frac{e^{- \frac{2 i \pi nQ}{c^2 (z+ d'/c + \ell)}} }{(z+d'/c + \ell)^2 |z+d'/c + \ell|^{2s}} = F_{c/\sqrt{Q},n,s} (z + d'/c)
\]
with the auxiliary function $F_{c,n,s}$ on $\Hcal$  defined for $c >0, n >0,s>0$ by 
\[
	F_{c,n,s} (z) := \sum_{\ell \in \Z} f_{c,n,s,z} (\ell), \quad {\textrm{with}} \quad f_{c,n,s,z}(x) := \frac{e^{- \frac{2 i \pi n}{c^2 (x+z)}}}{(x+z)^2 |x+z|^{2s}}.
\]
We will now give another expression for $F_{c,n,s}$ allowing us to compute more precisely the terms of $P_n (\cdot, s,N)_{|W_Q}$. 

As $f_{c,n,s,z}$ is $\Ccal^\infty$ on $\R$ and integrable as its two first derivatives, we can apply Poisson summation formula to rewrite 
\[
F_{c,n,s}(z) = \sum_{m \in \Z} \int_{- \infty}^{+ \infty} f_{c,n,s,z} (x) e^{-  2 i \pi m x} dx.
\]
Let us fix for now $\eta>0$, and restrict to the domain $\Im z \geq \eta$. The function $f_{c,n,s,z}$ then extends to an holomorphic function on $|\Im x| < \eta $ when we use the usual determination of the logarithm on $\C \backslash \R^-$ to write, for $x \in \R$,
\[
(x+z)^2 |x+z|^{2s} = (x+z)^{2+s} (x+ \overline{z})^s.
\]
The right term is clearly holomorphic in $x$ (when $z$ is fixed) so we can extend it on the domain $|\Im x| < \eta$. Notice that we still have $\left| (x+\overline{z})^s \right| = |x+\overline{z}|^s$ by definition of the determination of logarithm. As $f_{c,n,s,z}$ is holomorphic on this domain, we can shift the imaginary part of the integration axis by $\varepsilon \eta/2$, with $\varepsilon = -1$ if $m>0$ and $\varepsilon = 1$ otherwise, so that 
\[
\Re (- 2 i \pi m x) = 2 \pi m \Im (x) = - \pi |m| \eta.
\]
We then have 
\begin{eqnarray*}
\left|\int_{- \infty}^{+ \infty} f_{c,n,s,z} (x) e^{ - 2 i \pi m x} dx \right|  & = & \left| \int_{i \varepsilon \eta/2 + \R} f_{c,n,s,z} (x) e^{ - 2 i \pi m x} dx  \right|	\\
& \leq & \int_{i \varepsilon \eta/2 + \R} \frac{e^{ -  \pi |m| \eta}}{|x+z|^{2 + 2s}} dx \\
& \leq & e^{ - \pi |m| \eta} \int_\R \frac{1}{(\eta^2 / 4 + x^2)^{1+s}} dx.
\end{eqnarray*}

By real translation in the integral, we also see that for every $y \in \R$ and every $m \in \Z$, 
\[
	\int_\R f_{c,n,s,z+y} (x) e^{- 2 i \pi m (x+y)} dx = \int_{\R} f_{c,n,s,z} (x) e^{- 2 i \pi m x} dx.
\]
We can then rewrite 
\begin{eqnarray*}
\frac{P_n (\cdot,s,N)_{|W_Q} (z)}{Q |j_{W_Q} (z)|^{2s}} & = &  \sum_{\substack{c>0 \\ N |c \\ (Q,c/Q) = 1}} \frac{1}{c^{2 + 2s}} \sum_{\substack{0 \leq a,d \leq c \\ Q|(a,d) \\ ad \equiv Q [c]}} e^{ 2 i \pi n a/c} \sum_{m \in \Z} \int_\R f_{c/\sqrt{Q},n,s,z+d/c} (x) e^{ - 2 i \pi m x} dx \\
	& = &   \sum_{\substack{c>0 \\ N |c \\ (Q,c/Q) = 1}} \frac{1}{c^{2 + 2s}} \sum_{\substack{0 \leq a,d \leq c \\ Q|(a,d) \\ ad \equiv Q [c]}} e^{ 2 i \pi (n a+ md)/c} \sum_{m \in \Z} \int_\R f_{c/\sqrt{Q},n,s,z} (x) e^{ - 2 i \pi m x} dx \\
	& = &  \sum_{\substack{c>0 \\ N |c \\ (Q,c/Q) = 1}} \frac{1}{c^{2 + 2s}}  \sum_{m \in \Z} \sum_{\substack{0 \leq a,d \leq c \\ Q|(a,d) \\ ad \equiv Q [c]}} e^{ 2 i \pi (n a+ md)/c} \int_\R f_{c/\sqrt{Q},n,s,z} (x) e^{ - 2 i \pi m x} dx.
\end{eqnarray*}
For a fixed $c$, the integers $a$ and $d$ go through the multiples of $Q$ between $0$ and $c$ such that $a d = Q \mod c$. This amounts to say that $a = Qa'$ and $d = Qd'$ where $a',d'$ go through the integers between 0 and $c/Q$ such that $Q a'd' = 1 \mod c$, i. e. $d'$ is equal to $Q^{-1} {a'}^{-1}$ modulo $c/Q$. This proves the equality 
\[
\sum_{\substack{0 \leq a,d \leq c \\ Q|(a,d) \\ ad \equiv Q [c]}} e^{ 2 i \pi (n a+ md)/c} = S(m,n Q^{-1}; c/Q)
\]
where $Q^{-1}$ is the inverse of $Q$ modulo $c/Q$, so 
\[
P_n (\cdot,s,N)_{|W_Q} (z)  =  Q |j_{W_Q} (z)|^{2s}\sum_{\substack{c>0 \\ N |c \\ (Q,c/Q) = 1}} \frac{S(m,nQ^{-1} ;c/Q)}{c^{2 + 2s}} \sum_{m \in \Z} \int_\R f_{c/\sqrt{Q},n,s,z} (x) e^{ - 2 i \pi m x} dx.  
\]
Now, using the Weil bounds \eqref{bornesdeWeilKloosterman} on Kloosterman sums: 
\[
	\left| \frac{S(m,nQ^{-1} ;c/Q)}{c^{2 + 2s}} \sum_{m \in \Z} \int_\R |f_{c/\sqrt{Q},n,s,z} (x) e^{ - 2 i \pi m x}| dx   \right| \leq \frac{n^{1/2} \tau (c/Q)}{c^{3/2}} e^{- \pi |m| \eta} \int_\R \frac{1}{\eta^2/4 + x^2} dx
\]
which is the general term of an absolutely convergent series, allowing us to exchange the sum and the integral in the expression of $P_n (\cdot,s,N)_{|W_Q} (z) $, hence
\[
P_n (\cdot,s,N)_{|W_Q} (z)   = Q |j_{W_Q} (z)|^{2s} \sum_{m \in \Z} \sum_{\substack{c>0 \\ N |c \\ (Q,c/Q) = 1}} \frac{S(m,nQ^{-1} ;c/Q)}{c^{2 + 2s}}  \int_\R f_{c/\sqrt{Q},n,s,z} (x) e^{ - 2 i \pi m x} dx.  
\]
We can also take the limit $s \ra 0^+$ of this equality as the previous bound of absolute convergence does not depend on $s$, therefore we obtain 
\begin{eqnarray*}
	P_n (\cdot,N)_{|w_Q} (z) & = &  Q \sum_{m \in \Z} \sum_{\substack{c>0 \\ N |c \\ (Q,c/Q) = 1}} \frac{S(m,nQ^{-1} ;c/Q)}{c^{2}} \int_\R f_{c/\sqrt{Q},n,0,z} (x) e^{ - 2 i \pi m x} dx \\
	& = & Q \sum_{m \in \Z} \left( \sum_{\substack{c>0 \\ N |c \\ (Q,c/Q) = 1}} \frac{S(m,nQ^{-1} ;c/Q)}{c^{2}} \int_\R \frac{e^{ - \frac{2 i \pi n}{c^2/Q (x+z)} - 2 i \pi m(x+z)}}{(x+z)^2}  dx \right) e^{ 2 i \pi m z}.
\end{eqnarray*}
Let us compute this integral. Define 
\[
	G_{m,n,c} (z) : = \int_\R \frac{e^{ - \frac{2 i \pi n}{c^2 (x+z)} - 2 i \pi m(x+z)}}{(x+z)^2}  dx = \int_{i \Im (z) + \R} \frac{e^{ - \frac{2 i \pi n}{c^2 y} - 2 i \pi my}}{y^2} dy.
\]
As the term to integrate is holomorphic on $\C^*$, we can integrate on any horizontal line of ordinate $\alpha >0$, hence $G_{m,n,c} (z)$ does not depend on $z$, we denote it by $G_{m,n,c}$. For $m \leq 0$, we have 
\begin{eqnarray*}
| G_{m,n,c}| & = & \left| \int_{i \alpha + \R} \frac{ e^{- \frac{2 i \pi n}{c^2 y} - 2 i \pi m y}}{y^2} dy \right| \\
& \leq & \int_{i \alpha + \R} \frac{e^{2 \pi m \alpha}}{|y|^2} dy  \leq \int_{i \alpha + \R} \frac{dy}{|y|^2}
\end{eqnarray*}
and this goes to 0 when $\alpha$ goes to $+ \infty$, so $G_{m,n,c} = 0$ when $m \leq 0$.

Now, for $m>0$, 
\begin{eqnarray*}
	G_{m,n,c} & = & \int_{i + \R} \frac{ e^{- \frac{2 i \pi n}{c^2 y} - 2 i \pi m y}}{y^2} dy \\
	& = & 2 i \pi m \int_{2 \pi m - i \infty}^{2 \pi m + i \infty}  \frac{ e^{w - \frac{4 \pi^2 m n}{c^2 w} }}{w^2} dw, \qquad w = - 2 i \pi m \\
	& = & 2 i \pi m J_1 \left( \frac{4 \pi \sqrt{mn}}{c} \right) \frac{ic}{\sqrt{mn}}
\end{eqnarray*}
because of the integral representation of $J_1$ (Definition \ref{KloostermanBessel}). We finally obtain 
\begin{equation}
\label{formuleGmncJ1}
	G_{m,n,c} = - 2 \pi c \sqrt{\frac{m}{n}} J_1 \left( \frac{4 \pi \sqrt{mn}}{c} \right),
\end{equation}
hence 
\begin{eqnarray*}
	P_n (\cdot,N)_{|W_Q} (z) &  = & - 2 \pi Q \sqrt{m/n} \sum_{m \geq 1} \sum_{\substack{c>0 \\ N |c \\ (Q,c/Q) = 1}} \frac{S(m,nQ^{-1} ;c/Q)}{c^{2}} c/\sqrt{Q} J_1 \left( \frac{4 \pi \sqrt{mn}}{c/\sqrt{Q}} \right) e^{2i \pi m z} \\
	& = & - 2 \pi \sqrt{m/n} \sum_{m \geq 1} \sum_{\substack{c>0 \\ (N/Q) |c \\(Q,c) = 1}} \frac{S(m,nQ^{-1} ;c)}{c \sqrt{Q}}J_1 \left( \frac{4 \pi \sqrt{mn}}{c\sqrt{Q}} \right) e^{2i \pi m z} 
\end{eqnarray*}
after reindexation of $c$ by $c/Q$, which finishes the proof.
\end{proof}

\section{Final computations and proof of Corollary \ref{cormain}}
\label{sectionfinale}

We can now regroup all our results to obtain an exact formula for $(a_1,L_\chi)_{dp^2}^{+_{p^2}, \textrm{new}}$ and then estimate the error terms.

For every $N \geq 1$, every divisor $Q$ of $N$ such that $(Q,N/Q) = 1$ and every $x>0$, define
\[
	A_ {N,Q} (x) = 2 \pi \sum_{n=1}^{+ \infty} \frac{\chi(n)}{\sqrt{n}} e^{- 2 \pi n / x} \sum_{\substack{c>0 \\ (N/Q) |c \\ (c,Q)=1}} \frac{S(1,nQ^{-1} ; c)}{c \sqrt{Q}} J_1  \left( \frac{4 \pi \sqrt{n}}{c \sqrt{Q}} \right)
\]
and 
\[
	B_{N,Q} (x) = 2 \pi \sum_{n=1}^{+ \infty} \frac{\chi(n)}{\sqrt{n}} e^{-2 \pi n x / (D^2 N)} \sum_{\substack{c>0 \\ (N/Q) |c \\ (c,Q)=1}} \frac{S(1,nQ^{-1} ; c)}{c \sqrt{Q}} J_1  \left( \frac{4 \pi \sqrt{n}}{c \sqrt{Q}} \right)
\]
(so that $B_{N,Q} (x) = A_{N,Q} (D^2 N/x)$). We recognize here terms appearing in Proposition \ref{propformulesdestraces}, summed as indicated by the approximate functional equation of Lemma \ref{lemsumLfchi} $(c)$. More precisely, by Lemma \ref{a1Lchinewdp2},
\[
(a_1,L_\chi)_{dp^2}^{+_{p^2}, \textrm{new}} = (a_1,L_\chi)_{dp^2}^{+_{p^2}} - \frac{1}{p-1} (a_1,L_\chi)_{dp}^{\chi(p)_p},
\]
and by \eqref{eqapproxLfchi}, for any $x>0$ : 
\begin{eqnarray}
\label{eqcasdp2}
	(a_1,L_\chi)_{dp^2}^{+_{p^2}} & = & \sum_{n \geq 1} \frac{\chi(n)(a_1,a_n)_{dp^2}^{+_{p^2}}}{n} e^{ - \frac{ 2 \pi n}{x}} - \sum_{n \geq 1} \frac{\chi(n)(a_1,a_n \circ w_{dp^2})_{dp^2}^{+_{p^2}}}{n} e^{ - \frac{2 \pi n x}{dp^2} }	\nonumber \\
	& = & \sum_{n \geq 1} \frac{\chi(n)(a_1,a_n)_{dp^2}^{+_{p^2}}}{n} e^{ - \frac{ 2 \pi n}{x}} - \sum_{n \geq 1} \frac{\chi(n)((a_1,a_n)_{dp^2}^{+_{p^2},+_d} - (a_1,a_n)_{dp^2}^{+_{p^2},-_d})}{n} e^{ - \frac{2 \pi n x}{dp^2} } \nonumber \\
	(a_1,L_\chi)_{dp^2}^{+_{p^2}} & = & 2 \pi e^{ - \frac{2 \pi}{x}} - 2 \pi (A_{dp^2,1} (x) + A_{dp^2,p^2} (x)) + 2 \pi (B_{dp^2,dp^2}(x) + B_{dp^2,d} (x))
\end{eqnarray}
(we used here Remark \ref{remformuletracesPeterssonrestreintecombinee}). Here is also implicitly appearing the sign of the functional equation : for example, for $(a_1,L_\chi)_p^{+}$ (which is 0 by Lemma \ref{lemsumLfchi} $(c)$), we would have no principal term such as $2\pi$, and only error terms. In the same fashion, we obtain that for any $x>0$,  
\begin{equation}
\label{eqcasdp}
	(a_1,L_\chi)_{dp}^{\chi(p)_p} = 2 \pi e^{- \frac{2 \pi}{x}} - 2 \pi (A_{dp,1} (x) + \chi(p) A_{dp,p} (x)) + 2 \pi \chi(p) (B_{dp,dp}(x) + B_{dp,d} (x)).
\end{equation}
Consequently, we only have to give good estimates for the $A_{N,Q} (x)$ and $B_{N,Q} (x)$ (simultaneously in $x$). The idea for those is that we will choose $x$ of the same order of magnitude as $D^2 N$, so that $B_{N,Q}(x)$ is very small (given its exponential factors), whereas $A_{N,Q} (x)$ is not too large. 

Therefore, Theorem \ref{thmmain} is a direct consequence of the following lemma (notice that the only cases of $Q=N$ appearing in \eqref{eqcasdp2} and \eqref{eqcasdp} are for $B_{N,N}(x)$, hence made up to be small with our choice of $x$).

\begin{lem}
	For any $N \geq 1$, any divisor $Q$ of $N$ such that $(Q,N/Q)=1$, any $x>0$ and any quadratic Dirichlet character $\chi$ of conductor $D$ prime to $N$, 
	\begin{eqnarray*}
		|A_{N,Q} (x)| & \ll & \frac{\sqrt{D} (\log(D)+1) (\log(N) + \log(x))^2 \tau(N) e^{ - \frac{2 \pi}{x}}}{N} + \delta_{Q=N} \frac{x e^{ - \frac{2 \pi}{x}} \tau(D)}{N D^{3/2}} \\
		|B_{N,Q} (x)| & \ll & \frac{\sqrt{D} (\log(D)+1) (\log(N) + \log(x))^2 \tau(N) e^{ - \frac{ 2 \pi x}{D^2 N}}}{N} + \delta_{Q=N} \frac{\sqrt{D}e^{ - \frac{ 2 \pi x}{D^2 N}} \tau(D)} {x}
	\end{eqnarray*}
Therefore, choosing $x = (D^2 N) \log(D^2 N)$, we obtain after simplification and use of natural bounds that for $Q \neq N$ : 
\[
|A_{N,Q} (x)| + |B_{N,Q}(x)| \ll \frac{\sqrt{D} (\log(D)+1)^3 \log(N)^2 \tau(N)}{N}
\]
with an absolute implied constant. Applied to $N=dp^2$, this gives us the error term of Theorem \ref{thmmain}.
\end{lem}

\begin{proof}
	As we remarked before, $B_{N,Q} (x) = A_{N,Q} (N/x)$ so it is enough to obtain a bound on $A_{N,Q} (x)$ for all $x>0$.
	The double sum defining $A_{N,Q} (x)$ is absolutely convergent (e.g. by Weil bounds \eqref{bornesdeWeilKloosterman}), so 
	\[
		A_{N,Q}(x) = 2 \pi \sum_{\substack{c>0 \\ (N/Q) |c \\ (c,Q)=1}} A_{N,Q,c}(x),
	\]
with 
\[
	A_{N,Q,c} (x) = \frac{1}{c \sqrt{Q}} \sum_{n=1}^{+ \infty} \frac{\chi(n)}{\sqrt{n}} e^{- 2 \pi n / x} \frac{S(1,nQ^{-1};c)}{c \sqrt{Q}} J_1  \left( \frac{4 \pi \sqrt{n}}{c \sqrt{Q}} \right).
\]
With Weil bounds \eqref{bornesdeWeilKloosterman} and the bound $|J_1(t)| \ll |t|$ for $t$ real, we obtain 
\begin{eqnarray}
\label{eqANQcWeil}
	|A_{N,Q,c} (x)| & \ll & \sum_{n=1}^{+ \infty} \frac{e^{- 2 \pi n / x}}{c^2 Q} \tau(c) \sqrt{c} \ll \frac{\tau(c)}{Q c^{3/2}}  \sum_{n=1}^{+ \infty} e^{- 2 \pi n / x} \nonumber \\
	|A_{N,Q,c} (x)| & \ll & \frac{x e^{- 2 \pi /x} \tau(c)}{Q c^{3/2}}.
\end{eqnarray}
On another side, if $c \neq D$, there is a natural cancellation in the terms defining $A_{N,Q,c} (x)$. To see this, note that if $c \neq D$, one can apply Polya-Vinogradov techniques to obtain that for every integers $K,K'$, 
\[
	\left| \sum_{n=K}^{K'} \chi(n) S(1,nQ^{-1};c) \right| \leq \frac{4 c \sqrt{D}}{\pi^2} ( \log(Dc) + 1.5) \ll c \sqrt{D} (\log(Dc) +1)
\]
(this inequality is proved in Lemma 5.9 of \cite{LeFourn1}). Defining 
\[
T(c,n) = \sum_{k=1}^n \chi(k) S(1,kQ^{-1};c) \quad \textrm{   and   }  \quad f_{T} (y) = \frac{J_1\left( \frac{4 \pi \sqrt{y} }{c \sqrt{Q}} \right)}{4 \pi \sqrt{y} / c \sqrt{Q}} e^{- 2 \pi y / x},
\]
we can write by Abel transform 
\begin{eqnarray*}
	A_{N,Q,c} (x) & = & \sum_{n=1}^{+ \infty} \chi(n) S(1,nQ^{-1};c) \frac{4 \pi f_T (n)}{c^2 Q} =  \frac{4 \pi}{c^2 Q} \sum_{n=1}^{+ \infty} T(c,n) (f_T(n) - f_T(n+1))
\end{eqnarray*}
so that 
\begin{eqnarray*}
	|A_{N,Q,c} (x)| & \ll & \frac{1}{c^2 Q} c \sqrt{D} (\log(Dc)+1) \sum_{n=1}^{+ \infty} |f_T(n) - f_T(n+1)|
\end{eqnarray*}
and by definition of the $f_T$, this sum is $e^{-2\pi/x}$ times less than the total variation of $J_1(y)/y$ on $]0, + \infty[$, which is finite, so we obtain 
\begin{equation}
\label{eqANQcPV}
	|A_{N,Q,c} (x)|  \ll  \frac{(\log(Dc)+1) \sqrt{D}}{c Q} e^{ - 2 \pi /x}.
\end{equation}
Notice that this bound is naturally almost uniform on $x$, but not convergent in $c$, as opposed to the bound obtained previously.

Comparing quickly \eqref{eqANQcWeil} and \eqref{eqANQcPV}, it is natural to choose \eqref{eqANQcPV} for $c < x^2$ (except if $c = D$) and \eqref{eqANQcWeil} for $c> x^2$. This gives  (omitting for now the possible term $c=D$)
\begin{eqnarray*}
	|A_{N,Q} (x)| & \ll & \sum_{\substack{c < x^2 \\ (N/Q)|c \\ (Q,c)=1}} \frac{(\log(Dc)+1) \sqrt{D}}{c Q} e^{- 2 \pi /x} + \sum_{\substack{c \geq x^2 \\ (N/Q)|c \\ (Q,c)=1}} \frac{x e^{- 2 \pi /x} \tau(c)}{Q c^{3/2}} \\ 
	& \ll & \frac{\sqrt{D}(\log(D)+1)}{N} \left( \log(N/Q) \log (x) + \log(x)^2 \right)e^{- 2 \pi /x}  \\
	& & + \frac{\tau(N/Q) \sqrt{Q/N} x e^{- 2 \pi /x}}{N} \frac{\log(x)}{x} \\
	& \ll & \frac{\sqrt{D}(\log(D)+1)}{N} \left( \log(N) + \log(x) \right)^2 + \frac{\tau(N) e^{ - 2 \pi /x} \log(x)}{N} \\
	& \ll & \frac{\sqrt{D}(\log(D)+1)}{N} (\log(N) + \log (x)) ^2 \tau(N) e^{- 2 \pi /x}.
\end{eqnarray*}
Finally, notice that the possible term $c=D$ can appear only if $Q=N$ because $(D,N)= 1$, and we apply \eqref{eqANQcWeil} to it hence the $\delta_{Q=N}$ terms in the Lemma.
\end{proof}

To conclude this paper, we prove how Theorem \ref{thmmain} implies Corollary \ref{cormain}.

\begin{proof}[Proof of Corollary \ref{cormain}]
\hspace*{\fill}

$(a)$ By Theorem \ref{thmmain} in case $\chi=1$, there is an eigenform $f \in S_2(\Gamma_0(dp^2))^{+_{p^2}}$ such that $L(f,1)$ is nonzero. By the famous result of Kolyvagin and Logachev (\cite{KolyvaginLogachev}, Theorem 0.3), this implies that the abelian variety $A_f$ associated to $f$ obtained as a quotient of $J(d,p)$ is of algebraic rank zero. 

$(b)$ By Theorem \ref{thmmain}, there is an eigenform $f \in S_2(\Gamma_0(dp^2))^{+_{p^2},\textrm{new}}$ such that $L(f \otimes \chi,1)$ is nonzero. Such an eigenform necessarily satisfies $f_{|w_d} = - \chi(d) f$ because if $f_{|w_d}= \chi(d) f$, $L(f,1)=0$ (Lemma \ref{lemsumLfchi} $(c)$ because $\chi(-1)=1$ here). By Theorem 0.3 of \cite{KolyvaginLogachev} applied to $f \otimes \chi$, the abelian variety $A_f$ obtained as a quotient of $J(d,p)$ then has its twist by $K/\Q$ relatively to $[-1]$ of algebraic rank zero (because this twist is isogenous over $\Q$ to the abelian variety $A_{f \otimes \chi}$), and the canonical quotient morphism $J(d,p) \rightarrow A_f$ satisfies $\pi \circ w_d = - \chi(d) \pi$ because $f_{|w_d}= - \chi(d)f$, therefore this twist is a rational quotient of $J(d,p,\chi)$. Indeed, let $i : J(d,p) \rightarrow J(d,p,\chi)$ be the twist isomorphism (i.e. defined over $K$ and such that $j^{\sigma} = \chi(d) j \circ w_d$) and $i : A_f \rightarrow A_f \otimes \chi$ (i.e. defined over $K$ such that $i^\sigma = - i$, and consider $\pi'_f : J(d,p,\chi) \rightarrow A_{f \otimes \chi}$ the natural quotient morphism making the diagram below commutative 
\[
	\xymatrix{J(d,p) \ar[r]^{\pi_f} \ar[d]_{j} & A_f \ar[d]^{i} \\ J(d,p,\chi) \ar[r]^{\pi'_f} & A_{f} \otimes \chi}.
\]
Therefore, $\pi'_f$ is defined over $K$ and 
\[
{\pi'_f}^\sigma = i^\sigma \circ \pi_f^\sigma \circ (j^{-1})^{\sigma} = - i \circ \pi_f \circ (\chi(d) w_d \circ j^{-1}) = (- \chi(d))^2 i \circ \pi_f \circ j^{-1}= \pi'_f,
\]
which proves that $A_f \otimes \chi$ is a rational quotient of $J(d,p,\chi)$.

$(c)$ This is a consequence of $(a)$ and $(b)$ coming from the fact that for $d$ prime to $p$ and squarefree, there is a rational isogeny 
	\begin{equation}
	\label{eqdecompjacnonsplit}
		J'(d,p) \rightarrow J(d,p) \oplus J_0(d)
	\end{equation}
	equivariant under the action of the Hecke algebra generated by the $T_{\ell}$, $\ell$ prime not dividing $dp$, and $w_d$ such as it can be defined naturally on both sides. For $d \in \{2,3,5,7,13\}$, $J_0(d)=0$ hence the result.
	
	This fact is due to \cite{Chen} for $d=1$ and cited in \cite{DarmonMerel}, and its  principle has been generalised to any $d$ by \cite{DeSmitEdixhoven}. For the details, one proof can be found in \cite{LeFournthese2} (Lemma I.6.2).
\end{proof}

\bibliographystyle{amsalpha}
\bibliography{bibliotdn}

\end{document}